\documentclass[a4paper,12pt]{article}
\usepackage{latexsym}
\usepackage{amsmath}
\usepackage{amssymb}
\usepackage{enumerate}
\usepackage{bm}
\usepackage{mathrsfs}
\usepackage{color}

\usepackage[anchorcolor=blue,%
 bookmarks=true,%
 bookmarksnumbered=true,%
 colorlinks=true,%
 citecolor=cyan,%
 linkcolor=blue,%
 dvipdfmx%
]{hyperref}
\usepackage{theorem}
\newtheorem{theorem}{Theorem}[section]
\newtheorem{proposition}[theorem]{Proposition}

\theorembodyfont{\rmfamily}
\newtheorem{proof}{\textmd{\textit{Proof.}}}

\newtheorem{remark}[theorem]{Remark}
\newtheorem{example}[theorem]{Example}

\makeatletter

\@addtoreset{equation}{section}
\makeatother

\newcommand{\qedd}{\hfill \Box}
\newcommand{\ve}{\varepsilon}

\newcommand{\lra}{\longrightarrow}
\newcommand{\e}{\mathrm{e}}
\newcommand{\g}{\mathrm{g}}

\newcommand{\R}{\ensuremath{\mathbb{R}}}

\newcommand{\cC}{\ensuremath{\mathcal{C}}}
\newcommand{\cI}{\ensuremath{\mathcal{I}}}

\newcommand{\fm}{\ensuremath{\mathfrak{m}}}
\newcommand{\sP}{\ensuremath{\mathsf{P}}}

\def\vol{\mathop{\mathrm{vol}}\nolimits}
\def\supp{\mathop{\mathrm{supp}}\nolimits}
\def\Ent{\mathop{\mathrm{Ent}}\nolimits}

\def\Ric{\mathop{\mathrm{Ric}}\nolimits}
\def\CD{\mathop{\mathrm{CD}}\nolimits}
\def\RCD{\mathop{\mathrm{RCD}}\nolimits}

\title{Quantitative estimates for the Bakry--Ledoux isoperimetric inequality.~II}
\author{Cong Hung MAI\thanks{
Department of Mathematics, Osaka University, Osaka 560-0043, Japan
({\sf hungmcuet@gmail.com}, {\sf s.ohta@math.sci.osaka-u.ac.jp})}
\and
Shin-ichi OHTA\footnotemark[1] \textsuperscript{,}\thanks{
RIKEN Center for Advanced Intelligence Project (AIP),
1-4-1 Nihonbashi, Tokyo 103-0027, Japan}}

\date{}
\begin{document}

\maketitle

%%%%%%%%%%%%%%%%%%%%%%%%%%%%%%%%%%%%

\begin{abstract}
Concerning quantitative isoperimetry for a weighted Riemannian manifold
satisfying $\Ric_{\infty} \ge 1$, we give an $L^1$-estimate exhibiting that
the push-forward of the reference measure by the guiding function
(arising from the needle decomposition) is close to the Gaussian measure.
We also show $L^p$- and $W_2$-estimates in the $1$-dimensional case.
\end{abstract}

\section{Introduction}%%%%%%%%%%
%%%%%%%%%%%%%%%%%%%%%%%%%%

This short article is devoted to several further applications of the detailed estimates
in \cite{MO} to quantitative isoperimetry.
In \cite{MO}, on a weighted Riemannian manifold $(M,g,\fm)$
(with $\fm=\e^{-\Psi} \vol_g$) satisfying $\fm(M)=1$ and $\Ric_{\infty} \ge 1$,
we investigated the stability of the \emph{Bakry--Ledoux isoperimetric inequality} \cite{BL}:
\begin{equation}\label{eq:BL}
\sP(A) \ge \cI_{(\R,\bm{\gamma})}\big( \fm(A) \big)
\end{equation}
for any Borel set $A \subset M$, where $\sP(A)$ is the perimeter of $A$,
$\bm{\gamma}(dx)=(2\pi)^{-1/2} \e^{-x^2/2} \,dx$ is the Gaussian measure on $\R$,
and $\cI_{(\R,\bm{\gamma})}$ is its \emph{isoperimetric profile} written as
\begin{equation}\label{eq:a}
\cI_{(\R,\bm{\gamma})}(\theta) =\frac{\e^{-a_{\theta}^2/2}}{\sqrt{2\pi}},
 \qquad \theta =\bm{\gamma}\big( (-\infty,a_{\theta}] \big).
\end{equation}
It is known by \cite[Theorem~18.7]{Mo} (see also \cite[\S 3]{Ma2})
that equality holds in \eqref{eq:BL} for some $A$ with $\theta=\fm(A) \in (0,1)$ if and only if
$(M,g,\fm)$ is isometric to the product of $(\R,|\cdot|,\bm{\gamma})$ and
a weighted Riemannian manifold $(\Sigma,g_{\Sigma},\fm_{\Sigma})$ of $\Ric_{\infty} \ge 1$.
Moreover, $A$ is necessarily of the form
$(-\infty,a_{\theta}] \times \Sigma$ or $[-a_{\theta},\infty) \times \Sigma$ (so-called a \emph{half-space}).
Then, the stability result \cite[Theorem~7.5]{MO} asserts that,
if equality in \eqref{eq:BL} nearly holds,
then $A$ is close to a kind of half-space in the sense that
the symmetric difference between them has a small volume.

The proof as well as the formulation of \cite[Theorem~7.5]{MO}
are based on the \emph{needle decomposition} paradigm (also called the \emph{localization}),
which was established by Klartag \cite{Kl} for Riemannian manifolds
and has provided a significant contribution specifically in the study of isoperimetric inequalities
(we refer to \cite{CM} for a generalization to metric measure spaces satisfying
the curvature-dimension condition, and to \cite{CMM} for a stability result).
The half-space we mentioned above is in fact a sub-level or super-level set
of the \emph{guiding function} arising in the needle decomposition
(see Section~\ref{sc:M} and \cite{MO} for more details).
The needle decomposition enables us to decompose a global inequality on $M$
into the corresponding $1$-dimensional inequalities on minimal geodesics in $M$
(called \emph{needles} or \emph{transport rays}).
Therefore, a more detailed $1$-dimensional analysis on needles
will furnish a better estimate on $M$.

The $1$-dimensional analysis in \cite{MO} is concentrated in Proposition~3.2 in it
(restated in Proposition~\ref{pr:psi} below), which gives a very detailed estimate
on the difference from the Gaussian measure $\bm{\gamma}$.
In this article, as an application of the analysis developed in \cite{MO},
we show an $L^1$-bound between $\bm{\gamma}$ and
the push-forward measure $u_* \fm$ of $\fm$ by the guiding function $u$:
\[ \|\rho \cdot \e^{\bm{\psi}_{\g}} -1\|_{L^1(\bm{\gamma})} \le C(\theta,\ve) \delta^{(1-\ve)/(9-3\ve)}, \]
where $u_* \fm=\rho \,dx$ and $\bm{\gamma}=\e^{-\bm{\psi}_{\g}} \,dx$
(see Theorem~\ref{th:L^1/M} for the precise statement).
In the $1$-dimensional case (on intervals), we also prove
an $L^p$-bound with the improved (and sharp) order $\delta^{1/p}$
(Proposition~\ref{pr:L^p/I}; see Example~\ref{ex:1D} for the sharpness)
and an estimate of the $L^2$-Wasserstein distance $W_2$ (Proposition~\ref{pr:W_2/I}).
The use of $L^p$ and $W_2$ (instead of the volume of the symmetric difference)
is inspired by stability results for the Poincar\'e and log-Sobolev inequalities
(e.g., \cite{BF,BGRS,CF,IK,IM}).
We refer to Remark~\ref{rm:relate} for some further related works and open problems.
\medskip

\noindent
{\it Acknowledgements.}
We are grateful to Emanuel Indrei, whose question on the $L^p$-estimate led us to write this paper.
CHM was supported by Grant-in-Aid for JSPS Fellows 20J11328.
SO was supported in part by JSPS Grant-in-Aid for Scientific Research (KAKENHI) 19H01786.

\section{Quantitative estimates on intervals}\label{sc:1D}%%%%%%%%%%
%%%%%%%%%%%%%%%%%%%%%%%%%%

We first consider the $1$-dimensional case (on intervals)
and establish quantitative stability estimates in terms of the $L^p$-norm and the $W_2$-distance.
The $L^1$-bound will be instrumental to study the Riemannian case in the next section.

\subsection{An $L^p$-estimate}\label{ssc:L^p}%%%%%%%%%%
%%%%%%%%%%%%%%%%%%%%%%%%%%

Throughout this section,
let $I \subset \R$ be an open interval equipped with a probability measure
$\fm=\e^{-\psi} \,dx$ such that $\psi$ is \emph{$1$-convex} in the sense that
\[ \psi\big( (1-t)x +ty \big) \le (1-t)\psi(x) +t\psi(y) -\frac{1}{2}(1-t)t|x-y|^2 \]
for all $x,y \in I$ and $t \in (0,1)$.
This means that $(I,|\cdot|,\fm)$ satisfies $\Ric_{\infty} \ge 1$
(or the curvature-dimension condition $\CD(1,\infty)$), and \eqref{eq:BL} holds.
The $1$-dimensional isoperimetric inequality is well investigated in convex analysis.
An important fact due to Bobkov \cite[Proposition~2.1]{Bob} is that an isoperimetric minimizer
can be always taken as a half-space of the form $(-\infty,a] \cap I$ or $[b,\infty) \cap I$.
Now we restate \cite[Proposition~3.2]{MO}, which is the source of all the estimates.
Recall that $\bm{\gamma}=\e^{-\bm{\psi}_{\g}} \,dx$ is the Gaussian measure.

\begin{proposition}[\cite{MO}]\label{pr:psi}
Fix $\theta \in (0,1)$ and suppose that
\begin{equation}\label{eq:a_theta}
\fm\big( (-\infty,a_{\theta}] \cap I \big) =\theta
\end{equation}
and
\begin{equation}\label{eq:deficit}
\e^{-\psi(a_{\theta})} \le \e^{-\bm{\psi}_{\g}(a_{\theta})} +\delta
\end{equation}
hold for sufficiently small $\delta>0$ $($relative to $\theta)$.
Then we have
\begin{equation}\label{eq:delta>}
\psi(x)-\bm{\psi}_{\g}(x)
 \ge \big( \psi'_+(a_{\theta})-a_{\theta} \big) (x-a_{\theta}) -C(\theta) \delta
\end{equation}
for every $x \in I$, and
\begin{equation}\label{eq:delta<}
\psi(x)-\bm{\psi}_{\g}(x)
 \le \big( \psi'_+(a_{\theta})-a_{\theta} \big) (x-a_{\theta}) +C(\theta) \sqrt{\delta}
\end{equation}
for every $x \in [S,T] \subset I$ such that
$\lim_{\delta \to 0}S=-\infty$ and $\lim_{\delta \to 0}T=\infty$,
where $\psi'_+$ denotes the right derivative of $\psi$ and
$C(\theta)$ is a positive constant depending only on $\theta$.
\end{proposition}

The first condition \eqref{eq:a_theta} means that $I$ is ``centered''
in comparison with $\bm{\gamma}$ which satisfies $\bm{\gamma}((-\infty,a_{\theta}])=\theta$
(as in \eqref{eq:a}).
Note also that $\e^{-\psi(a_{\theta})} \ge \e^{-\bm{\psi}_{\g}(a_{\theta})}$ holds
by the isoperimetric inequality \eqref{eq:BL}
(since $\sP((-\infty,a_{\theta}] \cap I)=\e^{-\psi(a_{\theta})}$),
and then \eqref{eq:deficit} tells that the \emph{deficit} of $(-\infty,a_{\theta}] \cap I$
in the isoperimetric inequality is less than or equal to $\delta$.

Besides the above proposition, we also need the following estimate in its proof
(see \cite[(3.9)]{MO}):
\begin{equation}\label{eq:psi'/e}
\limsup_{\delta \to 0} \frac{|\psi'_+(a_{\theta}) -a_{\theta}|}{\delta}
 \le C(\theta).
\end{equation}
The lower bound \eqref{eq:delta>} enables us to
obtain the following $L^p$-estimate between
$\bm{\gamma}=\e^{-\bm{\psi}_{\g}} \,dx$ and $\fm=\e^{\bm{\psi}_{\g}-\psi} \,\bm{\gamma}|_I$.
(We remark that the upper bound \eqref{eq:delta<} will not be used.)

\begin{proposition}[An $L^p$-estimate on $I$]\label{pr:L^p/I}
Assume \eqref{eq:a_theta} and \eqref{eq:deficit}.
Then we have
\[ \| \e^{\bm{\psi}_{\g}-\psi} -1 \|_{L^p(\bm{\gamma})} \le C(p,\theta) \delta^{1/p} \]
for all $p \in [1,\infty)$ and sufficiently small $\delta>0$ $($relative to $\theta$ and $p)$,
where we set $\e^{\bm{\psi}_{\g}-\psi}:=0$ on $\R \setminus I$.
\end{proposition}

\begin{proof}
In this proof, we denote by $C$ a positive constant depending on $\theta$,
and put $a:=a_{\theta}$ for brevity.
Since $\e^{\bm{\psi}_{\g}-\psi} -1 \ge -1$ and $\fm(I)=\bm{\gamma}(\R)=1$, we find
\begin{align*}
\| \e^{\bm{\psi}_{\g}-\psi} -1 \|_{L^p(\bm{\gamma})}^p
&= \int_I \big[ \e^{\bm{\psi}_{\g}-\psi} -1 \big]_+^p \,d\bm{\gamma}
 +\int_{-\infty}^{\infty} \big[ 1-\e^{\bm{\psi}_{\g}-\psi} \big]_+^p \,d\bm{\gamma} \\
&\le \int_I \big[ \e^{\bm{\psi}_{\g}-\psi} -1 \big]_+^p \,d\bm{\gamma}
 +\int_{-\infty}^{\infty} \big[ 1-\e^{\bm{\psi}_{\g}-\psi} \big]_+ \,d\bm{\gamma} \\
&= \int_I \big[ \e^{\bm{\psi}_{\g}-\psi} -1 \big]_+^p \,d\bm{\gamma}
 +\int_I \big[ \e^{\bm{\psi}_{\g}-\psi} -1 \big]_+ \,d\bm{\gamma},
\end{align*}
where $[r]_+:=\max\{r,0\}$.
Thus, we need to estimate only $[ \e^{\bm{\psi}_{\g}-\psi} -1]_+$.
Observe that
\[ \big[ \e^{(\bm{\psi}_{\g}-\psi)(x)} -1 \big]_+^p
 \le \big( \e^{C\delta |x-a| +C\delta}-1 \big)^p
 \le \e^{p(C\delta |x-a| +C\delta)}-1 \]
from \eqref{eq:delta>} and \eqref{eq:psi'/e}, and hence
\begin{align*}
\int_I \big[ \e^{\bm{\psi}_{\g}-\psi} -1 \big]_+^p \,d\bm{\gamma}
&\le \int_{-\infty}^{\infty} \big( \e^{p(C\delta |x-a| +C\delta)}-1 \big) \,\bm{\gamma}(dx) \\
&= \frac{\e^{pC\delta}}{\sqrt{2\pi}}
 \int_{-\infty}^{\infty} \exp\left( -\frac{x^2}{2}+pC\delta |x-a| \right) \,dx -1.
\end{align*}
Dividing the integral into $(-\infty,a]$ and $[a,\infty)$,
we continue the calculation as
\begin{align*}
&\int_{-\infty}^a \exp\left( -\frac{x^2}{2}-pC\delta (x-a) \right) \,dx
 +\int_a^{\infty} \exp\left( -\frac{x^2}{2}+pC\delta (x-a) \right) \,dx \\
&= \int_{-\infty}^a
 \exp\left( -\frac{(x+pC\delta)^2}{2} +\frac{(pC\delta)^2}{2} +pCa\delta \right) \,dx \\
&\quad +\int_a^{\infty}
 \exp\left( -\frac{(x-pC\delta)^2}{2} +\frac{(pC\delta)^2}{2} -pCa\delta \right) \,dx \\
&\le \exp\left( \frac{(pC\delta)^2}{2} +pCa\delta \right)
 \left\{ \int_{-\infty}^a \e^{-x^2/2} \,dx +pC\delta \right\} \\
&\quad +\exp\left( \frac{(pC\delta)^2}{2} -pCa\delta \right)
 \left\{ \int_a^{\infty} \e^{-x^2/2} \,dx +pC\delta \right\} \\
&\le \exp\left( \frac{(pC\delta)^2}{2} +pC|a|\delta \right)
 \big( \sqrt{2\pi} +2pC\delta \big).
\end{align*}
Therefore, we obtain
\begin{align*}
\int_I \big[ \e^{\bm{\psi}_{\g}-\psi} -1 \big]_+^p \,d\bm{\gamma}
&\le \exp\left( pC\delta +pC|a|\delta +\frac{(pC\delta)^2}{2} \right)
 \left( 1 +\frac{2pC\delta}{\sqrt{2\pi}} \right) -1 \\
&\le C(p,\theta)\delta.
\end{align*}
This completes the proof.
$\qedd$
\end{proof}

We remark that, since
\[ \left\{ \exp\left( pC\delta +\frac{(pC\delta)^2}{2} \right) -1 \right\}^{1/p}
 \ge \exp\left( C\delta +\frac{p(C\delta)^2}{2} \right) -1, \]
the constant $C(p,\theta)$ given by the above proof necessarily depends on $p$.
The order $\delta^{1/p}$ in Proposition~\ref{pr:L^p/I} may be compared with
$L^p$-estimates in \cite{IK} for the log-Sobolev inequality on Gaussian spaces.
One can see that the order $\delta^{1/p}$ is optimal from the following example.

\begin{example}\label{ex:1D}
Let $I=(-D,D)$ and $\fm=(1+\delta) \cdot \bm{\gamma}|_I$,
where $\delta>0$ is given by $\bm{\gamma}(I)=(1+\delta)^{-1}$.
Then, at $\theta=1/2$, we have $a_{1/2}=0$, $\fm((-\infty,0] \cap I)=1/2$,
\[ \e^{-\psi(0)} -\e^{-\bm{\psi}_{\g}(0)} =\frac{\delta}{\sqrt{2\pi}}, \]
and
\[ \| \e^{\bm{\psi}_{\g}-\psi} -1 \|_{L^p(\bm{\gamma})}
 =\left( \frac{\delta^p}{1+\delta} +\frac{\delta}{1+\delta} \right)^{1/p}
 =\left( \frac{1+\delta^{p-1}}{1+\delta} \right)^{1/p} \delta^{1/p}. \]
\end{example}

\subsection{A $W_2$-estimate}\label{ssc:W_2}%%%%%%%%%%
%%%%%%%%%%%%%%%%

From Proposition~\ref{pr:psi},
one can also derive an upper bound of the \emph{$L^2$-Wasserstein distance}
between $\fm$ and $\bm{\gamma}$.
We refer to \cite{Vi} for the basics of optimal transport theory.
What we need is only the following \emph{Talagrand inequality}
with $\bm{\gamma}$ as the base measure (see \cite{Ta}, \cite[Theorem~22.14]{Vi}):
\begin{equation}\label{eq:Tala}
W_2^2(\fm,\bm{\gamma})
 \le 2\Ent_{\bm{\gamma}}(\fm)
 = 2\int_I (\bm{\psi}_{\g} -\psi)\e^{\bm{\psi}_{\g} -\psi} \,d\bm{\gamma},
\end{equation}
where $\Ent_{\bm{\gamma}}(\fm)$ is the \emph{relative entropy} of $\fm$
with respect to $\bm{\gamma}$.
We remark that both $\bm{\gamma}$ and $\fm$ have finite second moment
(by the $1$-convexity of $\psi$).

\begin{proposition}[A $W_2$-estimate on $I$]\label{pr:W_2/I}
Assume \eqref{eq:a_theta} and \eqref{eq:deficit}.
Then we have
\[ W_2(\fm,\bm{\gamma}) \le C(\theta) \sqrt{\delta} \]
for sufficiently small $\delta>0$ $($relative to $\theta)$.
\end{proposition}

\begin{proof}
We again denote $a_{\theta}$ by $a$,
and $C$ will be a positive constant depending only on $\theta$.
Similarly to the proof of Proposition~\ref{pr:L^p/I},
we observe from \eqref{eq:delta>} and \eqref{eq:psi'/e} that
\begin{align*}
\int_I (\bm{\psi}_{\g} -\psi)\e^{\bm{\psi}_{\g} -\psi} \,d\bm{\gamma}
&\le \int_{-\infty}^{\infty} (C\delta|x-a|+C\delta) \e^{C\delta|x-a|+C\delta} \,\bm{\gamma}(dx) \\
&= \frac{C\delta}{\sqrt{2\pi}} \e^{C\delta}
 \int_{-\infty}^{\infty} (|x-a|+1) \exp\left( -\frac{x^2}{2}+C\delta|x-a| \right) \,dx \\
&\le C\delta \left\{
 \int_{-\infty}^{\infty} |x-a| \exp\left( -\frac{x^2}{2}+C\delta|x-a| \right) \,dx +C \right\},
\end{align*}
where we used
\[ \int_{-\infty}^{\infty} \exp\left( -\frac{x^2}{2}+C\delta|x-a| \right) \,dx \le C \]
from the proof of Proposition~\ref{pr:L^p/I}.
Then we have
\begin{align*}
&\int_{-\infty}^a (a-x) \exp\left( -\frac{x^2}{2}-C\delta (x-a) \right) \,dx \\
&= \exp\left( Ca\delta +\frac{(C\delta)^2}{2} \right)
 \int_{-\infty}^a (a-x) \exp\left( -\frac{(x+C\delta)^2}{2} \right) \,dx \\
&\le (1+C\delta)
  \left\{ (a+C\delta) \int_{-\infty}^a \exp\left( -\frac{(x+C\delta)^2}{2} \right) \,dx
 +\left[ \exp\left( -\frac{(x+C\delta)^2}{2} \right) \right]_{-\infty}^a \right\} \\
&\le (1+C\delta)
 \left\{ a\int_{-\infty}^a \e^{-x^2/2} \,dx +C\delta
 +\exp\left( -\frac{(a+C\delta)^2}{2} \right) \right\} \\
&\le a\int_{-\infty}^a \e^{-x^2/2} \,dx +\e^{-a^2/2} +C\delta.
\end{align*}
We similarly find
\begin{align*}
&\int_a^{\infty} (x-a) \exp\left( -\frac{x^2}{2}+C\delta (x-a) \right) \,dx \\
&= \exp\left( -Ca\delta +\frac{(C\delta)^2}{2} \right)
  \int_a^{\infty} (x-a) \exp\left( -\frac{(x-C\delta)^2}{2} \right) \,dx \\
&\le (1+C\delta)
 \left\{ (-a+C\delta) \int_a^{\infty} \exp\left( -\frac{(x-C\delta)^2}{2} \right) \,dx
 -\left[ \exp\left( -\frac{(x-C\delta)^2}{2} \right) \right]_a^{\infty} \right\} \\
&\le (1+C\delta)
 \left\{ -a\int_a^{\infty} \e^{-x^2/2} \,dx +C\delta
 +\exp\left( -\frac{(a-C\delta)^2}{2} \right) \right\} \\
&\le -a\int_a^{\infty} \e^{-x^2/2} \,dx +\e^{-a^2/2} +C\delta.
\end{align*}
Therefore, together with the Talagrand inequality \eqref{eq:Tala},
we obtain the desired estimate $W_2^2(\fm,\bm{\gamma}) \le C\delta$.
$\qedd$
\end{proof}

We do not know whether the order $\sqrt{\delta}$ in Proposition~\ref{pr:W_2/I} is optimal.
Since $W_p(\fm,\bm{\gamma}) \le W_2(\fm,\bm{\gamma})$ for any $p \in [1,2)$
by the H\"older inequality, we have, in particular,
a bound of the $L^1$-Wasserstein distance:
\[ W_1(\fm,\bm{\gamma}) \le C(\theta) \sqrt{\delta}. \]
One can alternatively infer this estimate from the Kantorovich--Rubinstein duality
(see \cite{Vi}); in fact,
\[ W_1(\fm,\bm{\gamma})
 \le \int_{-\infty}^{\infty} |x-a| \!\cdot\! |\e^{(\bm{\psi}_{\g}-\psi)(x)} -1| \,\bm{\gamma}(dx)
 \le C(\theta) \sqrt{\delta}. \]
%%% MEMO at the end %%%

We also remark that, when we take a detour
via the reverse Poincar\'e inequality in \cite[Proposition~5.1]{MO}
and the stability result \cite[Theorem~1.2]{CF},
we arrive at a weaker estimate
\[ W_1(\fm,\bm{\gamma}) \le C(\theta,\ve) \delta^{(1-\ve)/4}. \]
We refer to \cite{CMS,FGS} for stability results for the Poincar\'e inequality
(equivalently, the spectral gap)
on $\CD(N-1,N)$-spaces and $\RCD(N-1,N)$-spaces with $N \in (1,\infty)$.

\section{An $L^1$-estimate on weighted Riemannian manifolds}\label{sc:M}%%%%%%%%%%
%%%%%%%%%%%%%%%%%%%%%%%%%%

Next, we consider a weighted Riemannian manifold,
namely a connected, complete $\cC^{\infty}$-Riemannian manifold $(M,g)$
of dimension $n \ge 2$ equipped with a probability measure $\fm=\e^{-\Psi} \vol_g$,
where $\Psi \in \cC^{\infty}(M)$ and $\vol_g$ is the Riemannian volume measure.
Assuming $\Ric_{\infty} \ge 1$, we have the Bakry--Ledoux isoperimetric inequality \eqref{eq:BL}.

We begin with an outline of the proof of \eqref{eq:BL}
via the \emph{needle decomposition} (see \cite{Kl}).
Given a Borel set $A \subset M$ with $\theta=\fm(A) \in (0,1)$,
we employ the function $f:=\chi_A -\theta$
($\chi_A$ denotes the characteristic function of $A$)
and an associated $1$-Lipschitz function $u:M \lra \R$
attaining the maximum of $\int_M f\phi \,d\fm$ among all $1$-Lipschitz functions $\phi$.
Then, analyzing the behavior of $u$, one can build a partition
$\{X_q\}_{q \in Q}$ of $M$ consisting of (the image of) minimal geodesics (called \emph{needles}),
and $Q$ is endowed with a probability measure $\nu$.
For $\nu$-almost every $q \in Q$, $u|_{X_q}$ has slope $1$
($|u(x)-u(y)|=d(x,y)$ for all $x,y \in X_q$)
and $X_q$ is equipped with a probability measure $\fm_q$
such that $\fm_q(A \cap X_q)=\theta$ and
$(X_q,|\cdot|,\fm_q)$ satisfies $\Ric_{\infty} \ge 1$.
Moreover, we have
\begin{equation}\label{eq:Q}
\int_M h \,d\fm =\int_Q \bigg( \int_{X_q} h \,d\fm_q \bigg) \,\nu(dq)
\end{equation}
for all $h \in L^1(\fm)$.
Then, \eqref{eq:BL} for $A$ is obtained by integrating its $1$-dimensional counterparts
for $A \cap X_q$ with respect to $\nu$.

The $1$-Lipschitz function $u$ is called the \emph{guiding function}.
We can assume $\int_M u \,d\fm=0$ without loss of generality,
and $X_q$ will be identified with an interval via $u$
(in other words, $X_q$ is parametrized by $u$).
Denote $\fm_q=\e^{-\sigma_q} \,dx$ and $\mu:=u_* \fm=\rho \,dx$.
Note that $\supp \mu$ is an interval and may not be the whole $\R$.
Through the parametrization of $X_q$ by $u$, we deduce from \eqref{eq:Q} that
\begin{equation}\label{eq:rho}
\rho(x) =\int_Q \e^{-\sigma_q(x)} \,\nu(dq),
\end{equation}
where we set $\e^{-\sigma_q(x)}:=0$ if $x \not\in X_q$.

\begin{theorem}[An $L^1$-estimate on $M$]\label{th:L^1/M}
Assume $\Ric_{\infty} \ge 1$ and fix $\ve \in (0,1)$.
If $\sP(A) \le \cI_{(\R,\bm{\gamma})}(\theta) +\delta$ holds for some Borel set $A \subset M$
with $\theta=\fm(A) \in (0,1)$ and sufficiently small $\delta$ $($relative to $\theta$ and $\ve)$,
then $u_* \fm=\rho \,dx$ satisfies
\[ \|\rho \cdot \e^{\bm{\psi}_{\g}} -1\|_{L^1(\bm{\gamma})} \le C(\theta,\ve) \delta^{(1-\ve)/(9-3\ve)}, \]
where $u$ is the guiding function associated with $A$
such that $\int_M u \,d\fm=0$.
\end{theorem}

\begin{proof}
First of all, by \eqref{eq:rho} and Fubini's theorem, we have
\[ \|\rho \cdot \e^{\bm{\psi}_{\g}} -1\|_{L^1(\bm{\gamma})}
 =\int_{-\infty}^{\infty}
 \left| \int_Q (\e^{\bm{\psi}_{\g}-\sigma_q}-1) \,\nu(dq) \right| d\bm{\gamma}
 \le \int_Q \| \e^{\bm{\psi}_{\g}-\sigma_q}-1 \|_{L^1(\bm{\gamma})} \,\nu(dq). \]
We shall estimate
$\| \e^{\bm{\psi}_{\g}-\sigma_q}-1 \|_{L^1(\bm{\gamma})}$
by dividing into ``good'' needles and ``bad'' needles.
Note that $\nu(Q_{\ell}) \ge 1-\sqrt{\delta}$ holds for
\[ Q_{\ell} :=\big\{ q \in Q \,\big|\, \fm_q(A \cap X_q)=\theta,\,
 \sP(A \cap X_q) <\cI_{(\R,\bm{\gamma})}(\theta) +\sqrt{\delta} \big\} \]
by \cite[Lemma~7.1]{MO},
where $ \sP(A \cap X_q)$ denotes the perimeter of $A \cap X_q$ in $(X_q,|\cdot|,\fm_q)$.
Moreover, it follows from \cite[Proposition~7.3]{MO} that
there exists a measurable set $Q_c \subset Q$ such that
$\nu(Q_c) \ge 1-\delta^{(1-\ve)/(9-3\ve)}$ and
\[ \max\big\{ |a_{\theta} -r_q^-|,|a_{1-\theta}-r_q^+| \big\}
 \le C(\theta,\ve) \delta^{(1-\ve)/(9-3\ve)} \]
for all $q \in Q_c \cap Q_{\ell}$, where
$\fm_q((-\infty,r_q^-] \cap X_q)=\fm_q([r_q^+,\infty) \cap X_q)=\theta$
(recall that $\bm{\gamma}((-\infty,a_{\theta}]) =\bm{\gamma}([a_{1-\theta},\infty)) =\theta$).

On the one hand, for $q \in Q_c \cap Q_{\ell}$,
note that either $\sP(A \cap X_q) \ge \e^{-\sigma_q(r_q^-)}$ or $\sP(A \cap X_q) \ge \e^{-\sigma_q(r_q^+)}$
holds by \cite[Proposition~2.1]{Bob} (recall Subsection~\ref{ssc:L^p}).
When $\sP(A \cap X_q) \ge \e^{-\sigma_q(r_q^-)}$, we put
\[ \bm{\gamma}_q(dx) =\e^{-\bm{\psi}_{\g,q}(x)} \,dx
 :=\e^{-\bm{\psi}_{\g}(x+a_{\theta}-r_q^-)} \,dx, \]
which is a translation of $\bm{\gamma}$ satisfying $\bm{\gamma}_q((-\infty,r_q^-])=\theta$.
Then, it follows from Proposition~\ref{pr:L^p/I}
(with $\e^{-\sigma_q(r_q^-)} \le \sP(A \cap X_q) \le \e^{-\bm{\psi}_{\g,q}(r_q^-)} +\sqrt{\delta}$)
and Cavalieri's principle that
\begin{align*}
\| \e^{\bm{\psi}_{\g}-\sigma_q} -1 \|_{L^1(\bm{\gamma})}
&\le \| \e^{\bm{\psi}_{\g,q}-\sigma_q} -1 \|_{L^1(\bm{\gamma}_q)}
 +\| \e^{-\bm{\psi}_{\g,q}} -\e^{-\bm{\psi}_{\g}} \|_{L^1(dx)} \\
&\le C(\theta) \sqrt{\delta} +2\frac{|a_{\theta}-r_q^-|}{\sqrt{2\pi}} \\
&\le C(\theta,\ve) \delta^{(1-\ve)/(9-3\ve)}.
\end{align*}
We have the same bound also in the case where $\sP(A \cap X_q) \ge \e^{-\sigma_q(r_q^+)}$
by reversing $I$ in Proposition~\ref{pr:L^p/I}.

On the other hand, for $q \in Q \setminus (Q_c \cap Q_{\ell})$,
we have the trivial bound
\[  \| \e^{\bm{\psi}_{\g}-\sigma_q}-1 \|_{L^1(\bm{\gamma})}
 \le \| \e^{\bm{\psi}_{\g}-\sigma_q} \|_{L^1(\bm{\gamma})} +\| 1 \|_{L^1(\bm{\gamma})}
 =2. \]
Therefore, we obtain
\[ \|\rho \cdot \e^{\bm{\psi}_{\g}} -1\|_{L^1(\bm{\gamma})}
 \le C(\theta,\ve) \delta^{(1-\ve)/(9-3\ve)} +2\big( 1-\nu(Q_c \cap Q_{\ell}) \big)
 \le C(\theta,\ve) \delta^{(1-\ve)/(9-3\ve)}. \]
$\qedd$
\end{proof}

Note that $q \in Q_c \cap Q_{\ell}$ is well-behaved and can be handled by the $1$-dimensional analysis,
whereas one has a priori no information of $q \in Q \setminus (Q_c \cap Q_{\ell})$.
This could be a common problem for stability estimates via the needle decomposition
(see, e.g., \cite[Theorem~6.2]{MO} showing
a reverse Poincar\'e inequality on a manifold from a sharper estimate on intervals).
In particular, it may be difficult to achieve the same order $\delta$
as in the $1$-dimensional case (Proposition~\ref{pr:L^p/I}) by the needle decomposition.
In the $L^p$-case, it is unclear (to the authors)
with what we can replace the trivial bound
$\| \e^{\bm{\psi}_{\g}-\sigma_q}-1 \|_{L^1(\bm{\gamma})} \le 2$.
For the Wasserstein distance $W_2$ or $W_1$,
we have the same problem on the control of $q \in Q \setminus (Q_c \cap Q_{\ell})$.

\begin{remark}[Further related works and open problems]\label{rm:relate}
\begin{enumerate}[(a)]
\item
Theorem~\ref{th:L^1/M} holds true also for reversible Finsler manifolds by the same proof
(see \cite[Remark~7.6(c)]{MO} and \cite{Oneedle,Obook}).

\item
As we mentioned in the introduction,
our $L^p$- and $W_2$-estimates are inspired by the quantitative stability
for functional inequalities.
We refer to \cite{BGRS,FIL,IK,IM} for the study of the \emph{log-Sobolev inequality}
on the Gaussian space:
\[ \Ent_{\bm{\gamma}}(f \bm{\gamma})
 \le \frac{1}{2} \mathrm{I}_{\bm{\gamma}}(f \bm{\gamma})
 =\frac{1}{2} \int_{\R^n} \frac{\|\nabla f\|^2}{f} \,d\bm{\gamma}, \]
where $\mathrm{I}_{\bm{\gamma}}(f \bm{\gamma})$ is the \emph{Fisher information}
of a probability measure $f\bm{\gamma}$ with respect to $\bm{\gamma}$.
They investigated the difference between $\bm{\gamma}$ and $f\bm{\gamma}$,
in terms of the additive deficit
$\delta(f)=\mathrm{I}_{\bm{\gamma}}(f\bm{\gamma})/2 -\Ent_{\bm{\gamma}}(f \bm{\gamma})$.
For instance, $W_2$-bounds (under certain convexity and concavity conditions on $f$)
were given in \cite{BGRS,IM}, and $L^1$- and $L^p$-bounds can be found in \cite{IK}.
In the setting of weighted Riemannian manifolds satisfying $\Ric_{\infty} \ge 1$
(as in Theorem~\ref{th:L^1/M}),
we have only the rigidity (see \cite{OT}) and the stability is an open problem.

\item
We have seen in \cite[\S 6]{MO} that
the reverse forms of the Poincar\'e and log-Sobolev inequalities
can be derived from the isoperimetric deficit.
The reverse Poincar\'e inequality then implies a $W_1$-estimate
for the push-forward by an eigenfunction thanks to \cite[Theorem~1.3]{BF}
(see also \cite{FGS}).
We also expect a direct $W_1$- or $W_2$-estimate for the push-forward by the guiding function,
which remains an open question (see \cite[Remark~7.6(g)]{MO}).

\item
Another direction of research is a generalization to negative effective dimension,
i.e., $\Ric_N \ge K>0$ with $N<-1$.
We have established rigidity in the isoperimetric inequality in \cite{Ma2},
thereby it is natural to consider quantitative isoperimetry,
though it seems to require longer calculations.
\end{enumerate}
\end{remark}

\end{document}